\documentclass[reqno]{amsart}

\usepackage[utf8]{inputenc}
\usepackage{a4wide, hyperref, microtype, url, amssymb, amsmath, pifont, graphicx, mathrsfs}
\usepackage{tikz, tikz-cd}
\input xy
\xyoption{all}
\usepackage[all,knot]{xy}
\usepackage{comment}

\usepackage{quiver}

\usepackage{amsthm}
\theoremstyle{plain}
\newtheorem{theorem}{Theorem}[section]
\newtheorem{corollary}[theorem]{Corollary}
\newtheorem{proposition}[theorem]{Proposition}

\theoremstyle{definition}
\newtheorem{definition}[theorem]{Definition}
\newtheorem{example}[theorem]{Example}
\newtheorem{remark}[theorem]{Remark}

\newtheorem{definition-lemma}[theorem]{Definition-Lemma}

\numberwithin{equation}{section}
\numberwithin{figure}{section}

\usepackage{todonotes}
\todostyle{mystyle}{color=green!30,size=\small}

\newcommand{\aff}{\operatorname{aff}}
\newcommand{\fin}{\operatorname{fin}}
\newcommand{\re}{\operatorname{re}}
\newcommand{\SL}{\operatorname{SL}}
\newcommand{\cf}{{\it cf }}

\newcommand{\ba}{\mathbf{a}}
\newcommand{\bd}{\mathbf{d}}
\newcommand{\bu}{\mathbf{u}}
\newcommand{\bZ}{\mathbb{Z}}

\newcommand{\cA}{\mathcal{A}}
\newcommand{\cT}{\mathcal{T}}

\title{On the growth of friezes via theta functions}

\author[Pierre-Guy Plamondon]{Pierre-Guy Plamondon}
\address[Pierre-Guy Plamondon]
{}
\email{pierre-guy.plamondon@uvsq.fr}
\urladdr{\url{https://plamondon.pages.math.cnrs.fr/plamondon/}}

\author[Salvatore Stella]{Salvatore Stella}
\address[Salvatore Stella]
{}
\email{salvatore.stella@univaq.it}
\urladdr{\url{http://people.disim.univaq.it/~salvatore.stella/}}

\begin{document}

\begin{abstract}
  We prove that the infinite friezes arising from the tubes of a given cluster algebra of acyclic affine type all have the same growth coefficients.
  Our proof uses identities satisfied by theta functions.
  This generalizes previous results in affine types~$ADE$ by several groups of authors.
\end{abstract}

\maketitle

\section{Infinite friezes}

In this short note, we prove (see Theorem~\ref{theo::main}) that the infinite friezes arising from the tubes of a given cluster algebra of affine type all have the same growth coefficient, thus generalizing results of \cite{BFPT} for affine type~$A$, \cite{BBGTY} for affine type~$D$, and \cite{BFPTY} for affine types~$ADE$.

An \emph{infinite frieze} is an array of positive integers of the form
\smallskip
\begin{center}
  \begin{tikzcd}[column sep=tiny, row sep=tiny, arrows=dash]
    \cdots \arrow[dr] & & a_{1,1} \arrow[dl] \arrow[dr] & & a_{1,2} \arrow[dl] \arrow[dr] & & a_{1,3} \arrow[dl] \arrow[dr] & & a_{1,4} \arrow[dl] \arrow[dr] & \\
    & a_{2,0} \arrow[dl] \arrow[dr] & & a_{2,1} \arrow[dl] \arrow[dr] & & a_{2,2} \arrow[dl] \arrow[dr] & & a_{2,3} \arrow[dl] \arrow[dr] & & \cdots \arrow[dl] \\
    \cdots \arrow[dr] & & a_{3,0} \arrow[dl] \arrow[dr] & & a_{3,1} \arrow[dl] \arrow[dr] & & a_{3,2} \arrow[dl] \arrow[dr] & & a_{3,3} \arrow[dl] \arrow[dr] & \\
    & \cdots & & \cdots & & \cdots & & \cdots & & \cdots
  \end{tikzcd}
\end{center}
satisfying the following unimodularity relation: $a_{i,j}a_{i,j+1} - a_{i-1, j+1}a_{i+1, j} = 1$, with the convention that~$a_{0,j} = 1$ for all~$j\in \bZ$.
In other words, all small diamonds form a~$2\times 2$ matrix with determinant~$1$.
Note that the array is infinite in the left, right and down directions, but not in the upward direction.

The upper row of an infinite frieze is called the~\emph{quiddity row}; it determines the whole frieze.
If it has period~$r$ then every row of the frieze does and we say that the infinite frieze is~$r$-periodic.

\begin{example}\label{exam::initial}
  Below are two infinite friezes.
  On the left is an infinite frieze that is~$2$-periodic (so we only show two entries per row); on the right is another that is~$3$-periodic.

  \begin{center}
    \begin{tikzcd}[column sep=tiny, row sep=0.8em, font=\tiny, arrows=dash]
      6 \arrow[dr] && 20 \arrow[dl] \arrow[dr] && && && && && 3 \arrow[dr] && 5 \arrow[dl] \arrow[dr] && 9 \arrow[dl] \arrow[dr]\\
      & 119 \arrow[dl] \arrow[dr] && 119 \arrow[dl]  && && && && && 14 \arrow[dl] \arrow[dr] && 44 \arrow[dl] \arrow[dr] && 26 \arrow[dl] \\
      2360 \arrow[dr] && 708 \arrow[dl] \arrow[dr] && && && && && 121 \arrow[dr] && 123 \arrow[dl] \arrow[dr] && 127 \arrow[dl] \arrow[dr]\\
      & 14041 \arrow[dl] \arrow[dr] && 14041 \arrow[dl] && && && && && 1063 \arrow[dl] \arrow[dr] && 355 \arrow[dl] \arrow[dr] && 591 \arrow[dl] \\
      83538 \arrow[dr] && 278460 \arrow[dl] \arrow[dr] && && && && && 5192 \arrow[dr] && 3068 \arrow[dl] \arrow[dr] && 1652 \arrow[dl] \arrow[dr]\\
      & 1656719 \arrow[dl] \arrow[dr] && 1656719 \arrow[dl]  && && && && && 14985 \arrow[dl] \arrow[dr] && 14277 \arrow[dl] \arrow[dr] && 14513 \arrow[dl] \\
      \cdots && \cdots && && && && && \cdots && \cdots && \cdots
    \end{tikzcd}
  \end{center}
\end{example}

For periodic infinite friezes, the following was observed in~\cite{BFPT}.

\begin{theorem}[Theorem 2.2 of~\cite{BFPT}]
  Let~$(a_{i,j})$ be an~$r$-periodic infinite frieze.
  Then the value of~$a_{r,j} - a_{r-2, j+1}$ is the same for all choices of~$j\in\bZ$.
\end{theorem}

\begin{definition}[\cite{BFPT}]
  Let~$(a_{i,j})$ be an infinite periodic frieze of minimal period~$r$.
  Its \emph{growth coefficients} are given by~$s_0 = 2$ and~$s_k = a_{rk, j} - a_{rk-2, j+1}$ for all~$k\geq 1$ (this is independent of the choice of~$j\in \bZ$).
  If we refer to a growth coefficient without specifying~$k$, then we always mean the one for~$k=1$.
\end{definition}

\begin{remark}
  It follows from~\cite[Proposition 2.10]{BFPT} that~$s_1$ determines all~$s_k$ for~$k\geq 2$ by the Chebyshev recurrence relation~$s_0=2$ and~$s_{k+2} = s_1s_{k+1} - s_k$ for~$k\geq 0$.
\end{remark}

\begin{example}
  The friezes in Example~\ref{exam::initial} both have growth coefficient~$s_1=118$.
\end{example}

In the proof of our main result, we will be concerned with friezes with values in rings other than~$\bZ$.
The above notions of~$r$-periodic infinite friezes and growth coefficients generalize readily; for the convenience of the reader, we nevertheless lay them out explicitly in the remainder of this section.

Let~$R$ be a commutative ring.
An \emph{infinite frieze with values in~$R$} is an array of elements~$a_{i,j}$ (with~$i\geq 0,\ j\in \bZ$) of~$R$ satisfying the unimodularity relation~$a_{i,j}a_{i,j+1} - a_{i-1, j+1}a_{i+1, j} = 1$, with~$a_{0,j} = 1$ for all~$j\in \bZ$.
An infinite frieze is~$r$-periodic if~$a_{i,j} = a_{i,j+r}$ for all~$i,j$.

The \emph{universal~(possibly $r$-periodic) infinite frieze} is the frieze~$(a_{i,j})$ with values in the polynomial ring~$\bZ[z_j \ | \ j\in \bZ]$ defined by
\[
  a_{i,j} = \det \begin{bmatrix} z_{j} & 1 & & & 0 \\
    1 & z_{j+1} & 1 && \\
    & \ddots & \ddots & \ddots & \\
    && 1 & z_{j+i-2} & 1 \\
    &&& 1 & z_{j+i-1}
  \end{bmatrix}_{i\times i}
\]
for~$i\geq 1,\ j\in \bZ$, (with~$j$ taken modulo~$r$ for the~$r$-periodic version).
The fact that this defines a frieze is a consequence of the Desnanot--Jacobi identity applied to the above matrix.
These determinants appeared for infinite friezes of positive integers in~\cite[Sect. 2]{BFPT}, in~\cite{Coxeter} for finite friezes, and are related to relations obtained in~\cite{BR} for~$\SL_k$-tilings of the plane.

The above matrix is tridiagonal, and as such its determinant is a \emph{continuant} (applied to the case where the upper and lower diagonal only contain~$1$'s).
We can use classical identities on continuants to prove what we need; we will refer to \cite[Ch. XIII]{Muir}.:

First, we exhibit the~$a_{i,j}$ as explicit polynomials in the~$z_i$.

\begin{proposition}
  \label{prop::explicit-formula}
  Consider the (non-periodic) universal infinite frieze defined above.
  Then
  \begin{enumerate}
    \item each~$a_{i,j}$ is the sum of the monomial~$z_jz_{j+1}\cdots z_{j+i-1}$ and of all monomials obtained from it by replacing disjoint pairs of consecutive~$z_k$'s by~$-1$;

    \item each~$a_{i,j}-a_{i-2,j+1}$ is the sum of the monomial~$z_jz_{j+1}\cdots z_{j+i-1}$ and of all monomials obtained from it by replacing disjoint pairs of consecutive~$z_k$'s in the product (allowing also the pair formed by~$z_{j+i-1}$ and~$z_j$) by~$-1$.
  \end{enumerate}
\end{proposition}

\begin{proof}
  The first point is proved in paragraphs 544 and 545 in \cite{Muir}.
  The second point follows directly from the first.
\end{proof}

\begin{example}
  We have that~$a_{5,1} = z_1z_2z_3z_4z_5 - z_1z_2z_3 - z_1z_2z_5 - z_1z_4z_5 - z_3z_4z_5 + z_1+z_3+z_5$ and $a_{5,1}-a_{3,2} = z_1z_2z_3z_4z_5 - z_1z_2z_3 - z_1z_2z_5 - z_1z_4z_5 - z_2z_3z_4 - z_3z_4z_5  + z_1+z_2+z_3+z_4+z_5$.
\end{example}

We can use these expressions to prove that growth coefficients are well defined for the universal periodic friezes.
The results of the next proposition have already appeared for friezes of integers in \cite{BFPT}, but we show that they hold for universal friezes and thus for friezes with non-zero entries in any integral domain.

\begin{proposition}
  \label{prop::unifrieze}
  Consider the universal~$r$-periodic infinite frieze.
  The following properties hold.
  \begin{enumerate}
    \item The \emph{growth coefficients}~$s_k (k\geq 0)$ are well-defined, that is, the value of~$s_k = a_{rk, j} - a_{rk-2, j+1}$ does not depend on~$j$.

    \item For all~$j\in \bZ, i\in \bZ_{>0}$ and~$1\leq \ell \leq i$, we have that~$a_{i,j} = a_{\ell,j}a_{i-\ell, \ell+j} - a_{\ell - 1,j}a_{i-\ell-1,\ell+j+1}$.

    \item We have that~$s_{k+2} = s_1s_{k+1} - s_k$ for~$k\geq 0$ (with the convention that~$s_0=2$).
  \end{enumerate}
\end{proposition}

\begin{proof}
  We use the expression of~$a_{i,j}-a_{i-2,j+1}$ in Proposition~\ref{prop::explicit-formula}, and note that if~$i$ is a multiple of~$r$, then it is invariant under the change of indices~$j\leftarrow (j+1)$ (since~$j$ is taken modulo~$r$).
  This proves point~(1).
  Point (2) is a consequence of paragraph 547, formula (1) of~\cite{Muir}.
  To prove point (3), first use (2) to write
  \begin{eqnarray*}
    s_{k+1}s_1 &=& (a_{r(k+1),1} - a_{r(k+1)-2,2})(a_{r,1}-a_{r-2,2}) \\
    &=& a_{r(k+1),1}a_{r,1} - a_{r(k+1),1}a_{r-2,2} - a_{r(k+1)-2,2}a_{r,1} + a_{r(k+1)-2,2}a_{r-2,2}.
  \end{eqnarray*}
  Next, use (2) and~$r$-periodicity to write
  \begin{eqnarray*}
    s_{k+2} &=& a_{r(k+2),1} - a_{r(k+2)-2,2} \\
    &=& a_{r(k+1),1}a_{r,1+r(k+1)} - a_{r(k+1)-1,1}a_{r-1,r(k+1)+2} + \\
    && \ - a_{r(k+1)-1,2}a_{r-1,2+r(k+1)-1} + a_{r(k+1)-2,2}a_{r-2,2+r(k+1)} \\
    &=& a_{r(k+1),1}a_{r,1} - a_{r(k+1)-1,1}a_{r-1,2} - a_{r(k+1)-1,2}a_{r-1,1} + a_{r(k+1)-2,2}a_{r-2,2}.
  \end{eqnarray*}
  Taking the difference, we get
  \begin{eqnarray*}
    s_{k+2}-s_{k+1}s_1 &=& a_{r(k+1),1}a_{r-2,2} - a_{r(k+1)-1,1}a_{r-1,2} + a_{r(k+1)-2,2}a_{r,1} - a_{r(k+1)-1,2}a_{r-1,1} \\
    &=& (a_{rk,1}a_{r,rk+1} - a_{rk-1,1}a_{r-1,rk+2}) a_{r-2,2}  + \\
    && \ - (a_{rk,1}a_{r-1,rk+1} - a_{rk-1,1}a_{r-2,rk+2})a_{r-1,2} + \\
    && \ + (a_{rk-1,2}a_{r-1,rk+1} - a_{rk-2,2}a_{r-2,rk+2}) a_{r,1} + \\
    && \ - (a_{rk-1,2}a_{r,rk+1} - a_{rk-2,2}a_{r-1,rk+2}) a_{r-1,1} \\
    &=& a_{rk,1}a_{r,1}a_{r-2,2} - \underline{a_{rk-1,1}a_{r-1,2} a_{r-2,2} } + \\
    && \ - a_{rk,1}a_{r-1,1}a_{r-1,2} + \underline{a_{rk-1,1}a_{r-2,2}a_{r-1,2}} + \\
    && \ + \underline{\underline{a_{rk-1,2}a_{r-1,1} a_{r,1}}}- a_{rk-2,2}a_{r-2,2} a_{r,1} + \\
    && \ - \underline{\underline{a_{rk-1,2}a_{r,1}a_{r-1,1}}} + a_{rk-2,2}a_{r-1,2}a_{r-1,1} \\
    &=& a_{rk,1}(a_{r,1}a_{r-2,2} - a_{r-1,1}a_{r-1,2}) - a_{rk-2,2}(a_{r-2,2}a_{r,1} - a_{r-1,2}a_{r-1,1}) \\
    &=& -a_{rk,1} + a_{rk-2,2} \\
    &=& -s_{k},
  \end{eqnarray*}
  where the second equality is proved using (2), the third using~$r$-periodicity, and the fifth using the unimodularity rule.
\end{proof}

We deduce results on growth coefficient for friezes without zero entries in an integral domain.

\begin{corollary}
  Let~$R$ be an integral domain, and let~$(a_{i,j})$ be an~$r$-periodic infinite frieze whose values are all non-zero.
  Assume that~$r$ is the minimal period of the frieze.
  The following hold.
  \begin{itemize}
    \item The frieze is determined by its quiddity row and is obtained from the universal~$r$-periodic frieze by specializing each~$z_j$ to~$a_{1,j}$.

    \item The \emph{growth coefficients}~$s_k (k\geq 0)$ are well-defined, that is, the value of~$s_k = a_{rk, j} - a_{rk-2, j+1}$ does not depend on~$j$.

    \item We have that~$s_{k+2} = s_1s_{k+1} - s_k$ for~$k\geq 0$ (with the convention that~$s_0=2$).
  \end{itemize}
\end{corollary}

\begin{proof}
  If a frieze has non-zero entries in an integral domain, then the quiddity row determines the whole frieze using the unimodularity rule.
  This proves the first point, and the other points follow from Proposition~\ref{prop::unifrieze}.
\end{proof}

\section{Infinite friezes from cluster algebras of affine type}

In this section, we endeavour to take the shortest path to the definition of infinite friezes arising from cluster algebras of affine type.
We will only recall what is strictly necessary, and refer the reader to the references in the text for the full story.

\subsection{Roots, cluster variables and tubes}

Let~$B$ be an acyclic skew-symmetrizable~$n\times n$ matrix whose Cartan counterpart (\cf \cite[Sect. 1.3]{FZ2}) is of affine type.
To~$B$, we associate a cluster algebra (with trivial coefficients) as in~\cite{FZ2}.
Let~$\Phi$ be the corresponding root system,~$\Phi^+$ be the subset of positive roots, and~$W$ be the corresponding Weyl group.
The signs in the matrix~$B$ define a Coxeter element~$c\in W$.
Let~$\Pi=\{\alpha_1, \ldots, \alpha_n\}$ be the set of simple roots and let~$\delta$ be the positive imaginary root closest to the origin.
We define~$\tau_c$ as in~\cite[Sect. 3]{RS20}; it is an automorphism of the set~$-\Pi\cup \Phi^+$.
Following~\cite{RS20}, we define~$\Phi_c^{\re}$ to be the union of the~$\tau_c$-orbits of positive roots that do not have full support.
The set~$\Phi_c^{\re}$ is in bijection with cluster variables:

\begin{theorem}[Theorem 1.2 of \cite{RS20}, \cf \cite{FZ2} for finite types.]
  Let~$B$ be of affine type as above, and let~$\cA$ be the corresponding cluster algebra.
  For any cluster variable~$X$, let~$(d_1^X, \ldots, d_n^X)$ be its~$\bd$-vector.
  Then the map sending~$X$ to~$\sum_{i=1}^n d_i^X\alpha_i$ is a bijection from the set of cluster variables of~$\cA$ to~$\Phi_c^{\re}$.
\end{theorem}
We will write~$X_\beta$ to denote the cluster variable associated to a root~$\beta\in\Phi_c^{\re}$ and set~$X_0=1$.

We now turn to the definition of \emph{truncated tubes}, which will be a convenient way to organize the finite~$\tau_c$-orbits of~$\Phi_c^{\re}$.
According to~\cite[Proposition 3.12]{RS20}, there are exactly~$n-2$ finite~$\tau_c$-orbits in~$\Phi_c^{\re}$.
According to~\cite[Proposition 4.4]{RS20b} and~\cite[Proposition 2.12]{RS20}, these form a finite number of disjoint subsets~$\cT$ with the following properties:
\begin{itemize}
  \item Each~$\cT$ is stable under~$c$ and all~$c$-orbits in~$\cT$ have the same period~$r$.
  \item There exists a root~$\beta_0\in \cT$ such that, if we write~$\beta_i = c^k\beta_0$ for all $k\in\bZ$, then all other roots in $\cT$ are of the form $\sum_{j=0}^{s} \beta_{i+j}$ for some $0\leq s < r-1$ and some $i\in\bZ$.\footnote{
    In~\cite{RS20}, this is stated as follows.
    Let~$\alpha_{\aff}$ be a choice of an affine simple root, and~$\Phi_{\fin}\subset\Phi$ be the corresponding finite parabolic root system.
    Then the set of roots in~$\Phi_{\fin}$ with finite~$c$-orbit is a disjoint union of finitely many root systems of Dynkin type~$A$.
    Our roots~$\beta_1, \ldots, \beta_{r-1}$ are the simple roots for one of these systems while~$\beta_r=\delta-(\beta_1+ \cdots + \beta_{r-1})$
  }
\end{itemize}

\begin{definition}
  The \emph{truncated tubes} of~$B$ are the sets~$\cT$ described above.
  The size~$r$ of all~$c$-orbits in~$\cT$ is the \emph{rank} of~$\cT$.
\end{definition}

We will depict truncated tubes as follows (the example shown is of rank~$4$).
{\small
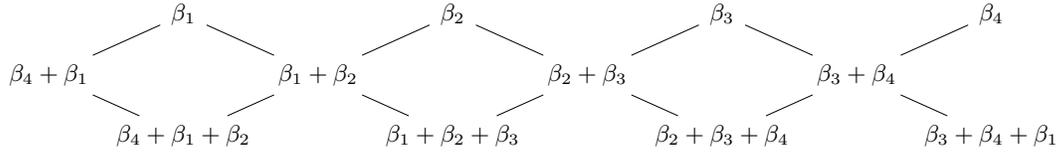
  \begin{center}
    \begin{tikzcd}[column sep=0.3em, row sep=0.9em, arrows=dash]
      & \beta_1 \arrow[dl] \arrow[dr] && \beta_2 \arrow[dl] \arrow[dr] && \beta_3 \arrow[dl] \arrow[dr] && \beta_4 \arrow[dl] \\
      \beta_4+\beta_1  \arrow[dr] && \beta_1+\beta_2 \arrow[dl] \arrow[dr] && \beta_2+\beta_3 \arrow[dl] \arrow[dr] && \beta_3+\beta_4 \arrow[dl] \arrow[dr] &\\
      & \beta_4+\beta_1+\beta_2 && \beta_1+\beta_2+\beta_3 && \beta_2+\beta_3+\beta_4 && \beta_3+\beta_4+\beta_1 &&  
    \end{tikzcd}
  \end{center}
}
The top row is called the \emph{mouth of the truncated tube}.
Note that each diamond
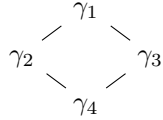
\begin{center}
  \begin{tikzcd}[column sep=tiny, row sep=tiny, arrows=dash ]
    & \gamma_1 \arrow[dl] \arrow[dr] & \\
    \gamma_2 \arrow[dr] & & \gamma_3\arrow[dl] \\
    & \gamma_4 &
  \end{tikzcd}
\end{center}
in the picture satisfies the relation~$\gamma_1+\gamma_4 = \gamma_2+\gamma_3$ (this includes the case where~$\gamma_2$ and~$\gamma_3$ are at the mouth of the tube, in which case~$\gamma_1=0$).
Note further that if we were to continue to fill another row by this diamond rule all its entries would equal~$\delta$.

\begin{example}\label{exam::f4-roots}
  Using the matrix~$B = \begin{bmatrix} 0&1&0&0&0 \\
    -1&0&2&0&0 \\
    0&-1&0&1&0 \\
    0&0&-1&0&1 \\
    0&0&0&-1&0
  \end{bmatrix}$ 
  of affine type~$F_4$, we get a truncated tube of rank~$2$ whose mouth is~$2\alpha_2 + \alpha_3 + \alpha_4, 2\alpha_1 + 2\alpha_2 + 2\alpha_3+\alpha_4+\alpha_5$ and one of rank~$3$ whose mouth is~$\alpha_1+\alpha_2+\alpha_3+\alpha_4, \alpha_2+\alpha_3, \alpha_1+2\alpha_2+\alpha_3+\alpha_4+\alpha_5$.
\end{example}

\begin{remark}
  (Truncated) tubes appear naturally in the representation theory of hereditary tame algebras, see for instance~\cite{SS} for affine types $ADE$ using representations of quivers, and~\cite{SY} for all affine types using representations of species.
\end{remark}

The following result will justify our definition of infinite friezes for cluster algebras of affine types in the next section.

\begin{proposition}
  Let~$\cT$ be a truncated tube, and let~$\gamma_1, \gamma_2, \gamma_3, \gamma_4\in \cT$ form a diamond as above.
  Then in~$\cA$ there is an exchange relation of the form~$X_{\gamma_2}X_{\gamma_3} = X_{\gamma_1}X_{\gamma_4} +1$.
\end{proposition}

\begin{proof}
  The claim in this proposition was established  in the case of principal coefficients and with different notations as \cite[Equation (4.3)]{RRS25} in the proof of \cite[Proposition 4.42]{RRS25}.
  Our case follows by evaluating all coefficients to 1.
\end{proof}

\subsection{Elements associated to imaginary roots}

We will need the following distinguished element of the cluster algebra~$\cA$.

\begin{definition-lemma}
  Let~$\beta_i$ be a root in the mouth of a truncated tube.
  Then the element
  \[
    X_\delta = X_{\beta_i} X_{\delta-\beta_i} - X_{\delta-\beta_i-\beta_{i-1}} - X_{\delta-\beta_i-\beta_{i+1}}
  \]
  is independent of the root~$\beta_i$ and of the truncated tube containing it.
\end{definition-lemma}

\begin{proof}
  The element $X_\delta$ is the \emph{theta basis element} whose denominator vector is $\delta$; despite the similarity in notation it is not a cluster variable in~$\cA$.
  The claimed identity was established in the case of principal coefficients and with different notation as \cite[Theorem~3.1]{RS26}.
  Our case follows by evaluating all coefficients to 1.
\end{proof}

\begin{definition-lemma}
  \label{deflem::highertheta}
  The theta basis elements~$X_{k\delta}$ whose~$\bd$-vector is~$k\delta$ satisfy the recursion
  \begin{align*}
    X_{2\delta} &= X_\delta^2 - 2 \\
    X_{(k+2)\delta} &= X_\delta X_{(k+1)\delta} - X_{k\delta}
  \end{align*}
  for~$k\geq1$.
\end{definition-lemma}

\begin{proof}
  Again this is a result of \cite{RS26}, namely Theorem~(3.3), adapted to the coefficient-free case.
\end{proof}

\subsection{Infinite friezes from cluster algebras of affine type}

\begin{definition}\label{defi::friezetube}
  Let~$B$ be an acyclic skew-symmetrizable~$n\times n$ matrix whose Cartan counterpart is of affine type.
  Let~$(\cT_i)_{i\in I}$ be the finite collection of truncated tubes of~$B$.
  For each~$i\in I$, let~$r_i$ be the rank of~$\cT_i$ and let~$X_{\beta_1^{(i)}}, \ldots, X_{\beta_{r_i}^{(i)}}$ be the cluster variables associated to the roots~$\beta_1^{(i)}, \ldots, \beta_{r_i}^{(i)}$ at the mouth of~$\cT_i$.

  Then we denote by~$\cT_i^{\bu}$ the infinite frieze obtained by taking the universal infinite frieze of rank~$r_i$ and specializing the variables~$z_1, \ldots, z_{r_i}$ at the cluster variables~$X_{\beta_1^{(i)}}, \ldots, X_{\beta_{r_i}^{(i)}}$.
  It is a frieze with values in~$\bZ[x_1^{\pm 1}, \ldots, x_n^{\pm 1}]$, where~$x_1, \ldots, x_n$ are the initial cluster variables.
\end{definition}

\begin{example}\label{exam::f4-variables}
  We continue Example~\ref{exam::f4-roots}.
  The cluster variables appearing at the mouth of the tubes are
  {\tiny
    \begin{eqnarray*}
      X_{2\alpha_2 + \alpha_3 + \alpha_4}
      &=& \frac{x_1^2x_2^2x_3 + x_1^2x_2^2x_5 + x_1^2x_4x_5 + 2x_1x_3x_4x_5 + x_3^2x_4x_5}{x_2^2x_3x_4}\\
      X_{2\alpha_1 + 2\alpha_2 + 2\alpha_3+\alpha_4+\alpha_5}
      &=& (x_1^2x_2^2x_3^2x_4x_5)^{-1}(x_1^2x_2^4x_3x_4 + x_1^2x_2^4x_3 + x_1^2x_2^2x_3x_4^2 + x_1^2x_2^4x_5 + 2x_1x_2^3x_3x_4x_5 + \\ && \ +x_2^2x_3^2x_4^2x_5 + x_1^2x_2^2x_3x_4 +  2x_1^2x_2^2x_4x_5 + 2x_1x_2^2x_3x_4x_5 + 2x_1x_2x_3x_4^2x_5 + 2x_2x_3^2x_4^2x_5 + \\ && \ +x_1^2x_4^2x_5 + 2x_1x_3x_4^2x_5 + x_3^2x_4^2x_5)\\
      X_{\alpha_1+\alpha_2+\alpha_3+\alpha_4}
      &=& \frac{x_1x_2^2x_3 + x_1x_2^2x_5 + x_2x_3x_4x_5 + x_1x_4x_5 + x_3x_4x_5}{x_1x_2x_3x_4}\\
      X_{\alpha_2+\alpha_3}
      &=& \frac{x_1x_2^2 + x_1x_4 + x_3x_4}{x_2x_3}\\
      X_{\alpha_1+2\alpha_2+\alpha_3+\alpha_4+\alpha_5}
      &=& \frac{x_1^2x_2^2x_3x_4 + x_1^2x_2^2x_3 + x_1^2x_2^2x_5 + x_1x_2x_3x_4x_5 + x_2x_3^2x_4x_5 + x_1^2x_4x_5 + 2x_1x_3x_4x_5 + x_3^2x_4x_5}{x_1x_2^2x_3x_4x_5}.
    \end{eqnarray*}
  }
  Specializing all the~$x_i$ to~$1$, we recover the two infinite friezes of Example~\ref{exam::initial}.
\end{example}

We can now state the main result of this note.

\begin{theorem}\label{theo::main}
  Let~$B$ be an acyclic skew-symmetrizable~$n\times n$ matrix whose Cartan counterpart is of affine type.
  Let~$(\cT_i)_{i\in I}$ the finite collection of tubes of~$B$.
  Then the friezes~$\cT_i^{\bu}$ all have growth coefficient~$X_\delta$.
\end{theorem}

Theorem~\ref{theo::main} recovers known results for friezes of positive integers in affine type~$A$ \cite[Theorem 3.4]{BFPT}, affine type~$D$~\cite[Theorem 5.5]{BBGTY}, and affine types~$ADE$ \cite{BFPTY}.

\begin{example}
  We continue with Example~\ref{exam::f4-variables}.
  The two friezes obtained from that cluster algebra of affine type~$F_4$ both have growth coefficient
  {\tiny
    \begin{eqnarray*}
      X_\delta &=& (x_1^2x_2^4x_3^3x_4^2x_5)^{-1}(x_1^4x_2^6x_3^2x_4 + x_1^4x_2^6x_3x_4x_5 + x_1^4x_2^6x_3^2 + x_1^4x_2^4x_3^2x_4^2 + 2x_1^4x_2^6x_3x_5 + 2x_1^3x_2^5x_3^2x_4x_5 + 2x_1^4x_2^4x_3x_4^2x_5 +  \\ && \ +2x_1^3x_2^4x_3^2x_4^2x_5 + x_1^4x_2^6x_5^2 +  2x_1^3x_2^5x_3x_4x_5^2 + x_1^2x_2^4x_3^2x_4^2x_5^2 + x_1^4x_2^4x_3^2x_4 + 4x_1^4x_2^4x_3x_4x_5 + 4x_1^3x_2^4x_3^2x_4x_5 +  \\ && \ +x_1^2x_2^4x_3^3x_4x_5 + 2x_1^3x_2^3x_3^2x_4^2x_5 + 2x_1^2x_2^3x_3^3x_4^2x_5 + x_1^4x_2^2x_3x_4^3x_5 + 2x_1^3x_2^2x_3^2x_4^3x_5 + x_1^2x_2^2x_3^3x_4^3x_5 + 3x_1^4x_2^4x_4x_5^2 + \\ && \ + 4x_1^3x_2^4x_3x_4x_5^2 + x_1^2x_2^4x_3^2x_4x_5^2 + 4x_1^3x_2^3x_3x_4^2x_5^2 + 6x_1^2x_2^3x_3^2x_4^2x_5^2 + 2x_1x_2^3x_3^3x_4^2x_5^2 + x_1^2x_2^2x_3^2x_4^3x_5^2 + 2x_1x_2^2x_3^3x_4^3x_5^2 + \\ && \ + x_2^2x_3^4x_4^3x_5^2 + 2x_1^4x_2^2x_3x_4^2x_5 + 4x_1^3x_2^2x_3^2x_4^2x_5 + 2x_1^2x_2^2x_3^3x_4^2x_5 + 3x_1^4x_2^2x_4^2x_5^2 + 8x_1^3x_2^2x_3x_4^2x_5^2 + 7x_1^2x_2^2x_3^2x_4^2x_5^2 +  \\ && \ +2x_1x_2^2x_3^3x_4^2x_5^2 + 2x_1^3x_2x_3x_4^3x_5^2 + 6x_1^2x_2x_3^2x_4^3x_5^2 + 6x_1x_2x_3^3x_4^3x_5^2 + 2x_2x_3^4x_4^3x_5^2 + x_1^4x_4^3x_5^2 + 4x_1^3x_3x_4^3x_5^2 + 6x_1^2x_3^2x_4^3x_5^2 +  \\ && \ +4x_1x_3^3x_4^3x_5^2 + x_3^4x_4^3x_5^2).
    \end{eqnarray*}
  }
  Specializing the~$x_i$ to~$1$, we recover~$118$, the growth coefficient of the friezes in Example~\ref{exam::initial}.
\end{example}

\begin{corollary}
  Keep the assumptions of Theorem~\ref{theo::main}.
  Let~$\ba = (a_1, \ldots, a_n)$ be an $n$-tuple of real numbers satisfying the following \emph{specialization condition}:
  \begin{description}
    \item[(*)] if we specialize the initial cluster variables to~$a_1, \ldots, a_n$, then the variables at the mouth of the tubes~$\cT_i$ become quiddity sequences for infinite friezes~$\cT_i^{\ba}$ of positive integers.
  \end{description}
  Then the friezes~$\cT_i^{\ba}$ all have the same growth coefficient, and that growth coefficient is the specialization of~$X_\delta$ at~$a_1, \ldots, a_n$.
\end{corollary}

\section{Proof of the main result}

Our main tool is the following special case of \cite[Theorem 3.8]{RS26}.

\begin{theorem}\label{theo::rs}
  Let~$\beta_i$ be a root in the mouth of a truncated tube.
  Then
  \[
    X_{\delta-\beta_{i-1}} X_{\delta-\beta_{i}} = X_{\delta-\beta_{i-1}-\beta_i}^2 + X_{\delta} X_{\delta-\beta_{i-1}-\beta_i} + 1.
  \]
\end{theorem}

\begin{proof}
  Again this is a specialization of a result proved in the principal coefficients case and with different notations.
  Theorem~3.8 of \cite{RS26} expands the product of any two cluster variables whose~$\bd$-vector sit in the~$(r-1)$-st row of a truncated tube of rank~$r$ in terms of theta functions but we only care about the case where the corresponding~$\bd$-vectors are part of a diamond.
\end{proof}

We can now prove the main result of this note.

\begin{proof}(of Theorem~\ref{theo::main}.)
  Let~$\cT$ be a truncated tube and let~$\beta_1, \ldots, \beta_r$ be the roots at the mouth of the tube.
  An entry on the~$(r-2)$-nd row of the frieze~$\cT^{\bu}$ has the form~$X_{\sum_{j=1}^{r-2}\beta_{i+j}} = X_{\delta - \beta_{i} - \beta_{i-1}}$ for some~$i\in\{1, \ldots, r\}$ (where the indices are taken modulo~$r$).
  This is a cluster variable, hence it is not~$0$ and we can divide by it.
  By Theorem~\ref{theo::rs}, we have that
  \[
    X_{\delta} = \frac{X_{\delta - \beta_i} X_{\delta- \beta_{i-1}} - 1}{X_{\delta - \beta_i - \beta_{i-1}}}  - X_{\delta - \beta_i - \beta_{i-1}}.
  \]
  Now, by the unimodularity rule,~$\frac{X_{\delta - \beta_i} X_{\delta- \beta_{i-1}} - 1}{X_{\delta - \beta_i - \beta_{i-1}}}$ is the entry in row~$r$ below~$X_{\delta - \beta_i - \beta_{i-1}}$ in the frieze.
  Thus the difference in the right-hand side of the above equation is the growth coefficient of the frieze~$\cT^{\bu}$.
  The equation thus says that this growth coefficient is~$X_\delta$.
\end{proof}

\begin{remark}
  By Proposition~\ref{prop::unifrieze} and Definition-Lemma~\ref{deflem::highertheta} the growth coefficients~$s_k$ and the theta functions~$X_{k\delta}$ satisfy the same recurrence relations so they coincide.
\end{remark}

\begin{remark}
  In the context of representation theory of hereditary algebras there is also the notion of \emph{homogeneous tubes}.
  These are tubes of rank~$1$ and the elements of the associated universal frieze satisfy the Chebyshev recursion with initial conditions~$a_{0,j} = 1$ and~$a_{1,j}=z$ for all $j\in\bZ$.
  The generalization to our setting is straightforward.
  As the Chebyshev recursion is linear, specializing such a frieze to~$z=X_\delta$, yields again a frieze whose growth coefficient is~$X_\delta$.
\end{remark}

\section{Acknowledgements}
We thank Karin Baur, Anna Felikson, Deepanshu Prasad, Pavel Tumarkin and Emine Y\i ld\i r\i m for letting us know of their work in \cite{BFPTY} where they analyse the growth coefficients of infinite frieze patterns arising from cluster algebras of tame type.
In the simply laced setting they derive the same results contained in this note using cluster modular groups and cluster categories.

Both authors were partially supported by a PHC Galileo grant.
The work of P.-G. P. is supported by the Institut Universitaire de France (IUF).
S.S. is partially supported by PRIN 20223FEA2E - PE1 and by INdAM - GNSAGA.

\bibliographystyle{plain}
\bibliography{sp}

\end{document}